\documentclass[12pt,a4paper]{article}
\usepackage[latin1]{inputenc}
\usepackage{amsmath, amsthm, amssymb, amsfonts, amscd}
\usepackage[english]{babel}
\usepackage[all]{xy}
\usepackage[linktocpage=true]{hyperref}
\usepackage{color}
\usepackage[table]{xcolor}
\usepackage{pinlabel}
\usepackage{tikz-cd}
 \usepackage{booktabs}
\usepackage{enumerate}
\usepackage{geometry}
\usepackage{mathtools}
\usepackage{authblk}
\usepackage{thmtools}
\usepackage{stmaryrd}
\usepackage{comment}

\declaretheoremstyle[bodyfont=,spaceabove=\medskipamount,
    spacebelow=\medskipamount]{definition}

\theoremstyle{definition}

\newtheorem{theorem}{Theorem}[section]
\newtheorem{lemma}[theorem]{Lemma}
\newtheorem{corollary}[theorem]{Corollary}
\newtheorem{proposition}[theorem]{Proposition}
\newtheorem{definition}[theorem]{Definition}

\newtheorem{conjecture}[theorem]{Conjecture}
\newtheorem{remark}[theorem]{Remark}

\newtheorem{example}[theorem]{Example}

\title{\large \textbf{THE MOCK ALEXANDER POLYNOMIAL FOR KNOTOIDS AND LINKOIDS}}

\author{\normalsize JOANNA A.~ELLIS-MONAGHAN, NESLIHAN G\"{U}G\"{U}MC\"{U}, LOUIS H.~KAUFFMAN, AND WOUT MOLTMAKER}

\date{}

\begin{document}

\maketitle

\begin{abstract}
The mock Alexander polynomial is an extension of the classical Alexander polynomial, defined and studied for (virtual) knots and knotoids by the second and third authors. In this paper we consider the mock Alexander polynomial for generalizations of knotoids. We prove a conjecture on the mock Alexander polynomial for knotoids, which generalizes to uni-linkoids. Afterwards we give constructions for canonical invariants of linkoids derived from the mock Alexander polynomial, using the formalism of generalized knotoids due to Adams et al.
\end{abstract}



\section{Context and Motivation}

We develop and then determine properties of a \emph{mock Alexander polynomial} which generalizes the Formal Knot Theory for the Alexander-Conway polynomial from \cite{kauffman1983formal} to knotoids and linkoids. This paper and its similarly focused predecessor \cite{MAP} thus have their roots in the Alexander polynomial for knots and links. 

When J. W. Alexander wrote his paper on that eponymous polynomial in 1928 \cite{Alex}, he understood that he was expressing a knot invariant that was related to the abelianization of the commutator subgroup of the knot group, that is, the fundamental group of the 
complement of the knot or link in three space,  viewed as a module over a Laurent polynomial ring ${ \mathbb Z}[t, t^{-1}]$.  Early last century, Alexander apparently did not have a satisfactory topological proof of the invariance of his polynomial. Such proofs came later with the concept of the infinite cyclic covering space of the knot complement \cite{Milnor}.  

Alexander made his definitions in terms of the combinatorial topology of knot diagrams and moves on the knot diagrams. Alexander and Briggs
\cite{AlexBriggs} had previously given three moves on knot diagrams, now known as the Reidemeister moves \cite{Reidemeister}, and they showed that two diagrams are equivalent by Reidemeister moves if and only if the embeddings in three dimensional space corresponding to the diagrams (via projection) are
ambient isotopic. Thus, Alexander and Briggs reformulated knot theory in terms of this move-based combinatorial topology of diagrams.

Accordingly, Alexander explains at the beginning of his paper \cite{Alex} how to associate a matrix $M_{K}(t)$ with entries in ${ \mathbb Z}[t, t^{-1}]$ so that the determinant of this matrix is the Alexander polynomial: $\Delta_{K}(t) \dot{=} \text{Det}(M_{K}(t)).$  The locution $\dot{=}$ indicates equality up to a factor of 
$\pm t^{n}$ for some integer $n.$ This expression of the polynomial up such a factor is related to the choices made in its definition. The reader can find an account of the structure of this matrix in the Alexander paper and also in the book \textit{Formal Knot Theory} \cite{kauffman1983formal}. In that book, Kauffman reformulated Alexander's definition as a state summation, by interpreting the terms of the determinant expansion graphically and locating an intrinsic graphical way to obtain the permutation signs. The state sum is then formulated both graphically and in terms of a permanent of an associated matrix. 

In \textit{Formal Knot Theory} \cite{kauffman1983formal} the state sum is reformulated to produce a model for the Alexander-Conway polynomial\cite{conway1970enumeration}, which is related to the Alexander polynomial by the formula
$\Delta_{K}(t) \dot{=} \nabla_{K}(z)$ where $z = \sqrt{t} - 1/\sqrt{t}.$  This invariant can be computed via a skein relation of the form
$\nabla_{K_{+}} - \nabla_{K_{-}} = z \nabla_{K_{0}}$.  Here $K_{+}$ denotes a diagram with a chosen positive crossing, $K_{-}$ denotes an identical diagram with this crossing switched to a negative crossing and $K_{0}$ denotes the same diagram with this crossing replaced by two non-crossing arcs (a {\it smoothing} of the 
crossing). For classical knot and link diagrams, the Alexander-Conway polynomial is determined by the normalization $\nabla_{O} = 1$ at the unknot $O,$ together with the requirement that if $K$ and $K'$ are related by Reidemeister moves, then $\nabla_{K} = \nabla_{K'}.$ One can use this state sum to prove many properties of the Conway-Alexander polynomial, and it is related to the categorification that occurs in Heegaard-Floer Link Homology \cite{baldwin2012combinatorial}.

The state sum description of the Alexander-Conway polynomial was extended to more general objects in the predecessor to this paper \cite{MAP}, where it was applied to define polynomial invariants, called \textit{mock Alexander polynomials}, of suitable diagrams for `starred' links and virtual links, as well as (starred) knotoids and linkoids. Knotoids were originally defined by V.~Turaev in \cite{turaev2012knotoids}, and have since received considerable interest. Knotoids have been studied in their own right as a natural generalization of links \cite{gugumcu2017new,barbensi2018double,gugumcu2021quantum,moltmaker2022framed}, and have also found application in modelling the topology of proteins \cite{dorier2018knoto,goundaroulis2020knotoids} and other systems of entangled filaments \cite{panagiotou2020knot,barkataki2022jones,barkataki2023virtual}. In this paper we study the mock Alexander polynomial of knotoids and linkoids to obtain new polynomial invariants of these objects.

The organization of the paper is as follows. In Section \ref{sec:prelim} we discuss the necessary preliminaries on knotoids, linkoids, and starred diagrams. All our invariants are then defined both as state summations and as permanents of certain matrices as described in Section \ref{sec:MAP}. In Section \ref{sec:MAP} we present a conjecture about the behaviour of these Mock Alexander polynomials from \cite{MAP}, and we prove the topological version of the conjecture using the skein relation in Section \ref{sec:knotoids}. Finally in Section \ref{sec:linkoids} we discuss the application of the state sum to linkoids, and describe a polynomial invariant of unstarred linkoids derived from the mock Alexander polynomial.

\section{Preliminaries}\label{sec:prelim}

\subsection{Knotoids and linkoids}

We begin with the definitions of knotoids and their generalizations, uni-linkoids and linkoids. Essentially, these are link diagrams that admit open-ended components whose endpoints may lie in any region of the diagram. The theory of knotoids was introduced in \cite{turaev2012knotoids}.

Let $I$ denote the unit interval, considered as an oriented manifold with boundary, and let $S^1$ denote the unit circle with an orientation. Let $\mathcal{L}_{\kappa,\ell} = \left( \bigsqcup_{j=1}^\kappa I \right) \sqcup \left( \bigsqcup_{i=1}^\ell S^1 \right)$ for some $\kappa,\ell \in \mathbb{Z}_{\geq 0}$.

\begin{definition}
Let $\Sigma$ be a surface. A \textit{$(\kappa,\ell)$-linkoid diagram} (or simply a linkoid diagram, if $\kappa$ and $\ell$ are clear or unspecified) in $\Sigma$ is an immersion $L$ of $\mathcal{L}_{\kappa,\ell}$, such that all the singularities of $L$ are transversal double points away from the boundary points of $\mathcal{L}_{\kappa,\ell}$. These double points are endowed with over/under-crossing information and called \textit{crossings} of $L$. 

The image 
of a unit interval under $L$ is called a \textit{knotoidal component} of $L$. The image of 
an oriented unit circle under $L$ is called a \textit{loop component} of $L$. The images of the boundary of $\mathcal{L}_{\ell,\kappa}$ are called the \textit{endpoints} of $L$. Specifically, the images of a copy of $0\in I$ and $1\in I$ are called the \textit{tail} and \textit{head} of the corresponding knotoidal component, respectively. Knotoidal components are oriented from tail to head.

A \textit{knotoid diagram} is a $(1,0)$-linkoid diagram. A \textit{uni-linkoid diagram} is a $(1,\ell)$-linkoid diagram with $\ell>0$.
\end{definition}

Throughout this paper we will restrict our attention to linkoid diagrams on $\Sigma=S^2$ and $\Sigma=\mathbb{R}^2$. There are called \textit{spherical} and \textit{planar} linkoid diagrams respectively. For examples of planar linkoids, see Figure \ref{fig:examples}. Note that braid- and tangle diagrams can be seen as (endpoint-labelled) spherical linkoid diagrams all of whose endpoints lie in a single region.

\begin{figure}[ht]
    \centering
    \includegraphics[width=.9\linewidth]{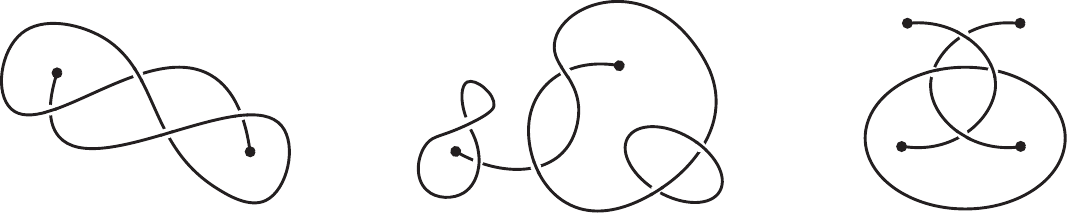}
    \caption{A knotoid diagram (left), uni-linkoid diagram (middle), and $(2,1)$-linkoid diagram (right).}
    \label{fig:examples}
\end{figure}

\begin{definition}
Two linkoid diagrams in $\Sigma$ are \textit{equivalent} if they can be related by isotopy of $\Sigma$ and a sequence of the familiar Reidemeister moves $R1$, $R2$, and $R3$. Note that the Reidemeister moves are local moves not involving the endpoints. As a result, the moves shown in Figure \ref{fig:forbiddenmoves} are not allowed. A \textit{linkoid} is an equivalence class of linkoid diagrams.
\end{definition}

\begin{figure}[ht]
    \centering
    \includegraphics[width=.7\linewidth]{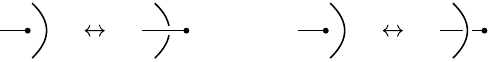}
    \caption{The `forbidden moves' for linkoid diagrams.}
    \label{fig:forbiddenmoves}
\end{figure}

\begin{remark}\label{rk:planar-spherical}
Unlike classical knots, the study of knotoids in $S^2$ differs from the study of knotoids in $\mathbb{R}^2$: Clearly the one-point compactification of $\mathbb{R}^2$ induces a surjection from planar to spherical knotoids, but this map turns out not to be injective. The `point at infinity' included in $S^2$ allows for more equivalence moves than planar isotopy, namely isotopies moving arcs across $\infty\in S^2$. Such moves are called \textit{spherical moves}, and clearly planar knotoids modulo spherical moves are equivalent to spherical knotoids. As a result there are pairs of knotoids that are inequivalent on $\mathbb{R}^2$ but become equivalent on $S^2$, the only difference between them being that their diagrams have $\infty \in S^2$ lying in different regions.
\end{remark}

\begin{definition}
Any knotoid in $S^2$ can be represented with a knotoid diagram in $\mathbb{R}^2$ whose tail lies in the exterior region of the plane. Such a diagram of a knotoid in $S^2$ is called its \textit{standard representation}.
\end{definition}

\begin{lemma}\cite{turaev2012knotoids} \label{lm:standard}
Two knotoid diagrams are equivalent to each other in $S^2$ if and only if their standard representations are equivalent in $\mathbb{R}^2$.
\end{lemma}

Lemma \ref{lm:standard} allows us to identify a knotoid in $S^2$ with its standard representation, and work with it as a diagram in (a portion of) the plane.

\begin{definition}
The \textit{universe} $U_L$ of a linkoid $L$ is the planar graph that is obtained by replacing any crossings and endpoints of $L$ with vertices (of degrees 4 and 1, respectively) and the strands connecting the vertices with edges. A \textit{region} of a universe $U_L$ is a connected component of the complement of $U_L$. The unbounded region of a universe $U_L$ in $\mathbb{R}^2$ is called the \textit{exterior region} of $U_L$.
\end{definition}

\begin{definition}\label{def:split}
A linkoid diagram is called \textit{split} if its universe is a disconnected graph. Otherwise, it is called \textit{connected}.
\end{definition}

Throughout this paper we assume all linkoid diagrams are connected. Note that up to equivalence, this is without loss of generality, as any split diagram can be made connected by $R2$ moves. The invariants we define are only defined for connected diagrams, so to show their invariance we make use of the following lemma:

\begin{lemma}\cite{MAP}\label{lm:split}
Let $L$ be a split linkoid diagram. Let $C_1,C_2$ be two connected linkoid diagrams each obtained from $L$ by selecting a pair of edges in $U_L$ that lie on the same face of $U_L$, applying an $R2$ move with the corresponding strands in $L$, and repeating until the resulting diagram is connected. Then $C_1$ and $C_2$ are equivalent by a series of Reidemeister moves moving only through connected diagrams.
\end{lemma}

\begin{proof}
Consider all the pairs of edges in $U_L$ used to obtain $C_2$ from $L$ that were not used to obtain $C_1$. Then use these edges to apply $R2$ moves to $C_1$. Next undo all the $R2$ moves applied to $L$ to obtain $C_1$, that weren't also applied to obtain $C_2$. The resulting diagram is $C_2$, and intermediate diagrams in this equivalence are clearly all connected.
\end{proof}

\begin{corollary}
An invariant of connected linkoids can be extended to split linkoids by defining the value of a split linkoid to be that of an equivalent connected linkoid.
\end{corollary}

A geometrical interpretation for knotoids in $S^2$ was given in \cite{turaev2012knotoids} in terms of `simple $\theta$-curves' in $S^3$. Later in \cite{gabrovvsek2023invariants} it was shown that similarly, linkoids can be interpreted as generalized $\theta$-graphs in $S^3$, of which theta-curves are a special case.

\begin{definition}\cite{gabrovvsek2023invariants}
A \emph{generalized} $\theta$-graph is an embedding in $S^3$ of a connected graph $G$ with an even number, say $2n$, of trivalent vertices, labeled $v_i$ and $w_i$ where $i \in \{1,\ldots,n\}$, as well as two distinguished degree $2n$ vertices $v_{\infty}$, $v_{-\infty}$.
The edge set $E(G)$ consists of edges $\{v_iw_i\}_{i=1}^n$, edges $\{v_i v_{\infty}\}_{i=1}^n$ and $\{v_i v_{-\infty}\}_{i=1}^n$ connecting $v_i$ to the vertices $v_{\infty}$ and $v_{-\infty}$, and edges 
 $\{w_i v_{\infty}\}_{i=1}^n$ and $\{w_i v_{-\infty}\}_{i=1}^n$ connecting $w_i$ to the points $v_{\infty}$ and $v_{-\infty}$.
Note that each vertex $v_i$ or $w_i$ is adjacent to both $v_{\infty}$ and  $v_{-\infty}$.

\end{definition}

\begin{definition}
A generalized $\theta$-graph is {\bf simple} if for every $i,j \in \{1,\ldots,n\}$, the spatial subgraph induced by edges $v_i v_{\infty}$, $v_i v_{-\infty}$, $w_j v_{\infty}$, and $w_j v_{-\infty}$ is an unknotted cycle. 
\end{definition}

\begin{theorem}\label{thm:linkoidsandtheta}\cite{gabrovvsek2023invariants}
There is a one-to-one correspondence between linkoids and simple generalized theta-graphs considered up to label-preserving ambient isotopy of $S^3$.
\end{theorem}

\begin{proof}
A full proof is given in \cite{gabrovvsek2023invariants}. For our purposes it suffices to describe the assignment constituting this proof here, without showing that it is indeed a bijection. The correspondence assigns a generalized theta-graph $\Theta(L)$ to a spherical linkoid diagram $L$ as follows:

Say $L$ has $2n$ endpoints. We see $L$ as being embedded in $S^3$ by perturbing its crossing arcs away from the embedding surface $S^2\subseteq S^3$ according to the crossing information. 
We label each knotoidal component of $L$ by $i \in \{1,\ldots,n\}$, and denote by $v_i$ the tail and by $w_i$ and the head of component $i$, such that component $i$ corresponds to an edge $v_iw_i$ in $\Theta(L)$. We also add two vertices, denoted by $v_{\infty}$ and $v_{-\infty}$ far above and below $L$, respectively.
Finally we add edges $v_i v_{\infty}$, $v_i v_{-\infty}$, $w_i v_{\infty}$, and $w_i v_{-\infty}$ that connect the vertices $v_i$ and $w_i$ to the vertices $v_{\infty}$ and $v_{-\infty}$ by straight line segments, for all $i \in \{1,...,2n\}$. See Figure \ref{fig:theta_graph} for an example of this assignment.

\end{proof}

\begin{figure}[ht]
\centering
\includegraphics[width=0.65\linewidth]{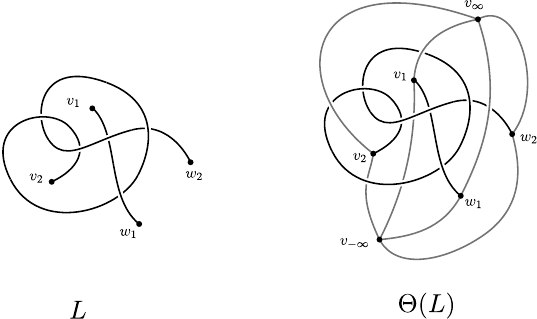}

\caption{A linkoid and the corresponding generalized $\Theta$-graph \cite{gabrovvsek2023invariants}.}

\label{fig:theta_graph}

\end{figure}

\subsection{Generalized knotoids}\label{subsec:generalized}

In studying linkoids, we will make use of another recent generalization of knotoids, namely `\textit{generalized knotoids}' \cite{adams2022generalizations}. We briefly introduce these here, as well as the starred linkoids defined in \cite{MAP} which form a subclass of generalized knotoids. Both of these generalizations are used in the subsequent sections to modify linkoid diagrams in order to make them amenable to the invariants we introduce in Section \ref{sec:MAP}.

Generalized knotoids are essentially link diagrams modelled on some graph, in which the graph's vertices (which may have arbitrary degree) are treated as knotoid endpoints in the sense that arcs are not allowed to be moved over or under them.

\begin{definition}
Let $\Sigma$ be a surface, $G$ a graph\footnote{We allow graphs with multiple edges, loops, and vertices of degree zero.}, and $\tilde{G}=G\sqcup \left( \bigsqcup_{i=0}^n S^1 \right)$ for some $n\in\mathbb{Z}_{\geq 0}$. A \textit{generalized knotoid diagram} is an immersion of $\tilde{G}$ in $\Sigma$ all of whose singularities are double points that occur away from the vertices of $\tilde{G}$ and are endowed with over/under-crossing data. The images of the vertices of $\tilde{G}$ are called the \textit{nodes} of the generalized knotoid. We consider generalized knotoid diagrams up to the usual Reidemeister moves and isotopy of $\Sigma$. Note that moves involving the nodes are forbidden, particularly the twist move and arc slide depicted in Figure \ref{fig:forbidden_vertex}. A \textit{generalized knotoid} is an equivalence class of generalized knotoid diagrams.
\end{definition}

\begin{figure}[ht]
    \centering
    \includegraphics[clip, trim=1cm 16cm 6cm 1cm, width=0.85 \linewidth]{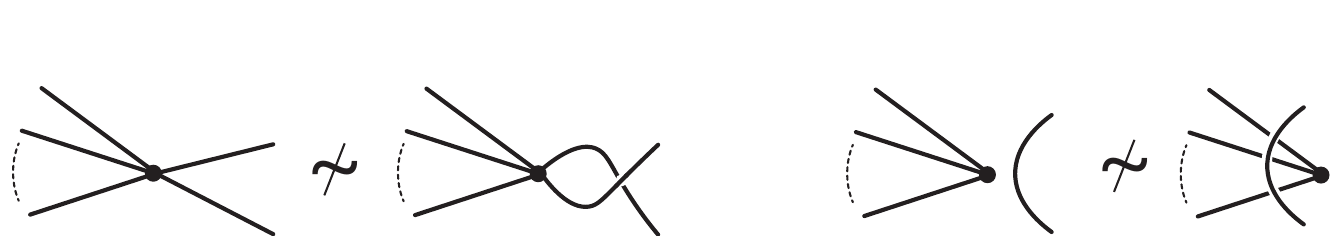}
    \caption{Forbidden moves at a node: a twist move (left) and an arc slide (right).}
    \label{fig:forbidden_vertex}
\end{figure}

Throughout this paper we will restrict our attention to generalized knotoids on $\Sigma=S^2$, and simply refer to these as `generalized knotoids'. See Figure \ref{fig:genknot_ex} for some examples of generalized knotoids. In \cite{MAP}, `starred' linkoid diagrams were considered, which can be seen as a subclass of generalized knotoids:

\begin{figure}[ht]
    \centering
    \includegraphics[width=0.8\linewidth]{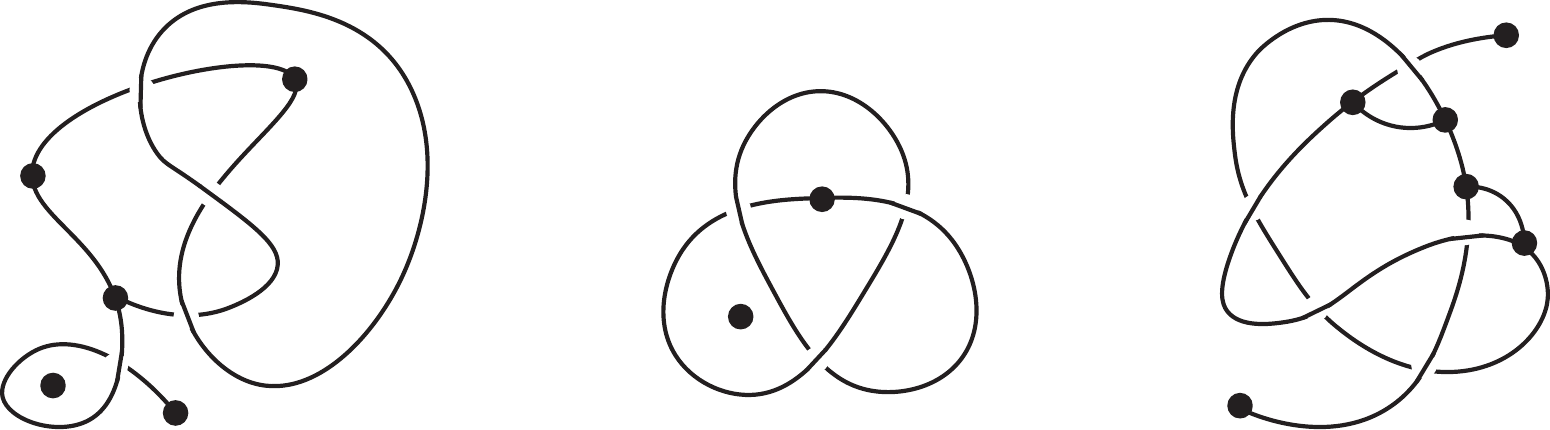}
    \caption{Examples of generalized knotoids.}
    \label{fig:genknot_ex}
\end{figure}

\begin{definition}
A \textit{starred linkoid diagram} is a linkoid diagram on $S^2$ some of whose regions and/or crossings have been decorated with a star. Two starred linkoid diagrams are said to be \textit{equivalent} if they can be related by surface isotopy and Reidemeister moves that do not move arcs across the stars. A \textit{starred linkoid} is an equivalence class of starred linkoid diagrams.
\end{definition}

See Figure \ref{fig:starred_ex} for two examples of starred linkoid diagrams. These two diagrams are not equivalent, as one has a starred crossing and the other does not.

\begin{figure}[ht]
    \centering
    \includegraphics[width=0.45\linewidth]{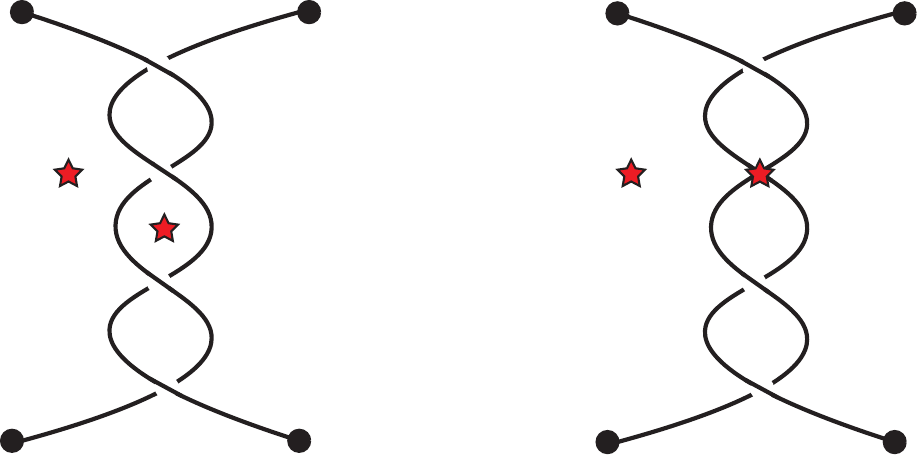}
    \caption{Two inequivalent starred linkoids.}
    \label{fig:starred_ex}
\end{figure}

\begin{example}\label{ex:planar_linkoids}
Planar linkoids are equivalent to a subclass of starred linkoids, namely those containing exactly one star in some region. This equivalence is given by interpreting the star in such a starred linkoid as the `point at infinity' in $S^2$ missing from $\mathbb{R}^2$. Clearly the equivalence of such starred linkoids carries over exactly to the equivalence of planar linkoids under this correspondence. Similarly linkoids on punctured surfaces can be seen as starred linkoids on compact surfaces.
\end{example}

\begin{remark}\label{rk:star_gen}
Starred linkoids are in turn clearly equivalent to a subclass of generalized knotoids, namely those that only have nodes of degree zero, one, or four. Under this equivalence a degree zero node corresponds to a star in a region, a degree one node to an endpoint, and a degree four node to a starred crossing. Note that the equivalence relations on generalized knotoids and starred linkoids also correspond with each other under this equivalence.
\end{remark}

Moreover using Example \ref{ex:planar_linkoids} one can also see planar linkoids as a subclass of generalized knotoids, namely those with only nodes of degree one and a single node of degree zero.

\section{The Mock Alexander Polynomial}\label{sec:MAP}

In this section we discuss the invariant that is the main subject of this article, the `\textit{mock Alexander polynomial}'. This polynomial invariant was defined in \cite{MAP}, and consists of an extension of the state sum formulation of the Alexander-Conway polynomial for links, which was first given in \cite{kauffman1983formal}.

\subsection{The state sum}\label{subsec:state_sum}

For the moment, assume that $L$ is a linkoid diagram such that the number of regions in $L$ is equal to its number of crossings. In this case, there is a bijection between regions and crossings. Every crossing is incident to four regions locally, which we call the quadrants adjacent to the crossings. These regions may or may not be distinct: for example these regions are not distinct if the crossing is nugatory or the crossing is adjacent to an endpoint of a knotoidal component. 

\begin{definition}
A \textit{state} of a linkoid diagram $L$ is a bijection between its regions and crossings such that each crossing is assigned to one of its adjacent quadrants.
\end{definition}

We represent a state $s$ of $L$ by a copy of the universe $U_L$ of $L$, with a small black wedge at each crossing pointing into the region associated to that crossing by $s$. See Figure \ref{fig:state_ex} for an example linkoid and a list of its states.

\begin{figure}[h]
    \centering
    \includegraphics[width=0.55\linewidth]{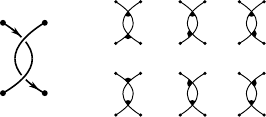}
    \caption{A spherical $(2,0)$-linkoid (left) and its states (right).}
    \label{fig:state_ex}
\end{figure}

We describe a state sum polynomial for linkoid diagrams, taking values in $\mathbb{Z}[W,B]$, by associating a weight to every state as follows:

\begin{definition}
Given a linkoid $L$, we assign local weights at each of its crossings by labelling the four quadrants adjacent to the crossing according to Figure \ref{fig:weights}. Given a state $s$ and a crossing $c$ of $L$ let $w_s(c)$ denote the label in the quadrant adjacent $c$ that is assigned to $c$ in the state $s$. We define the \textit{weight} $\langle L \vert s \rangle$ of a state $s$ by
\[
    \langle L \vert s \rangle = \prod_{c\in \mathcal{C}(L)} w_s(c) \in \mathbb{Z}[W,B],
\]
where $\mathcal{C}(L)$ denotes the set of all crossings in $L$.
\end{definition}

\begin{figure}[ht]
    \centering
    \includegraphics[width=0.35\linewidth]{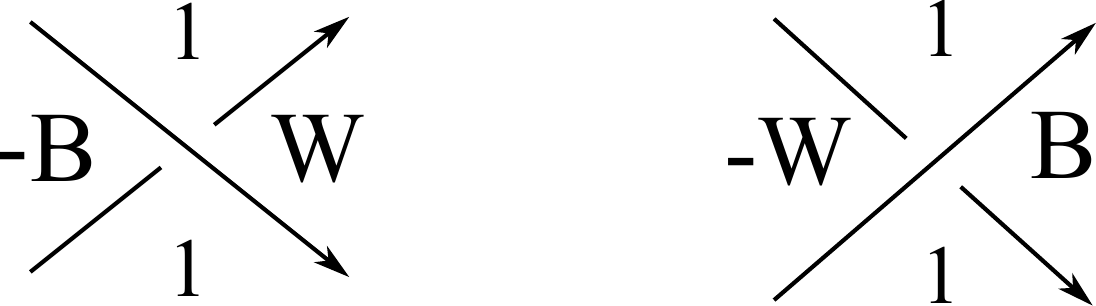}
    \caption{The local weight contributions at a positive crossing (left) and at a negative crossing (right).}
    \label{fig:weights}
\end{figure}

\begin{definition}\label{def:potential}
Let $L$ be a connected linkoid diagram with equal numbers of regions and crossings. Then the \textit{potential} of $L$, denoted $\nabla_L$, is defined by
\[
    \nabla_L(W,B) = \sum_{s\in \mathcal{S}(L)} \langle L \vert s\rangle,
\]
where $\mathcal{S}(L)$ is the set of states of $L$.
\end{definition}

\begin{remark}
The potential is not an invariant of linkoids, but as is shown in \cite{MAP} it becomes one after the substitution $B=W^{-1}$. We return to this point after a discussion of the requirement that there be equally many regions and crossings.
\end{remark}

\begin{definition}
Let $L$ be a linkoid diagram, and let $n$ and $f$ be its numbers of crossings and regions, respectively. We say that $L$ is \textit{admissible} if $f=n$.
\end{definition}

Most linkoid diagrams are not admissible, although the example in Figure \ref{fig:state_ex} is. To quantify the non-admissibility we define the obstruction of a linkoid as follows.

\begin{definition}
Let $L$ be a linkoid diagram, and $n,f$ as before. We define the \textit{obstruction} of $L$ to be $\Omega(L)=n-f$.
\end{definition}

Clearly $L$ is admissible if and only if $\Omega(L)=0$. The obstruction of a connected linkoid diagram is easily found using Euler's formula, $v-e+f=2$, for connected planar graphs:

\begin{proposition} \cite{MAP}\label{prop:obstruction}
Let $L$ be a connected spherical linkoid diagram with $\kappa$ knotoidal components. Then $\Omega(L)=\kappa-2$.
\end{proposition}

\begin{corollary}
If $L$ is a connected linkoid diagram in $S^2$, then $\Omega(L) =0$ if and only if $\kappa=2$.
\end{corollary}

Proposition \ref{prop:obstruction} implies that $\Omega(L)$ is a linkoid invariant, assuming that we restrict our attention to connected linkoids. The assumption of connectivity here is important: If two components of the universe of a split link are connected by an $R2$-move, two crossings are created and only one region. Therefore carrying out such a connecting move reduces $\Omega(L)$ by one, and $\Omega(L)$ is not an invariant of split linkoids. Consequently we do not define $\nabla_L$ for split linkoids, but restrict our attention to connected linkoids by appealing to Lemma \ref{lm:split}.

There are other ways of adjusting $\Omega(L)$ by changing the diagram of a linkoid slightly. The most direct way, discussed at length in \cite{MAP}, is to add stars to a linkoid diagram $L$, replacing it with a starred linkoid. We then proclaim that these starred crossings and regions don't contribute to the state sum, or more precisely they cannot be selected by a state.

\begin{definition}\label{def:states_starred}
Let $L$ be a starred linkoid diagram, and let $n_\star$ and $f_\star$ be its numbers of \textit{unstarred} crossings and regions, respectively. The \textit{obstruction} of $L$ is defined to be $\Omega_\star(L)=n_\star-f_\star$. $L$ is said to be \textit{admissible} if $\Omega_\star(L)=0$. A \textit{state} of an admissible starred linkoid is a bijective assignment of an adjacent unstarred region to each unstarred crossing. If $L$ is admissible, the \textit{potential} $\nabla_L(W,B)$ of $L$ is defined by the same state sum in Definition \ref{def:potential}.
\end{definition}

By replacing a linkoid diagram $L$ with a starred diagram $L'$, one can obviously always adjust $\Omega(L)$ to be zero. In other words, any linkoid diagram can be made admissible by endowing a number of its regions or crossings by stars. Specifically if $L$ is a knotoid diagram, $\Omega(L)=-1$ by Proposition \ref{prop:obstruction}. This means that endowing one of the regions of $L$ with a star turns $L$ into an admissible diagram. Example \ref{ex:one} gives the calculation of $\nabla_L(W,B)$ for two starred knotoid diagrams.

\begin{example}\label{ex:one}
Let $K_1$ be the starred knotoid depicted in Figure \ref{fig:labeltwo}, whose star is depicted as a degree zero node in accordance with Remark \ref{rk:star_gen}. Figure \ref{fig:labeltwo} depicts all the states of $K_1$ as well as their weights. It follows that
\[
    \nabla_{K_1}(W,B) = -W^3 + W^2B -WB^2 + B^3 + B^4.
\]

Similarly Figure \ref{fig:allstates2} depicts another starred knotoid $K_2$, which is identical to $K_1$ except for the starred region, as well as its states and their weights. From these we find
\[
    \nabla_{K_2}(W,B) = W^4 - W^3 + W^2B - WB^2 + B^3
\]

\begin{figure}[ht]
\centering
\includegraphics[width=\textwidth]{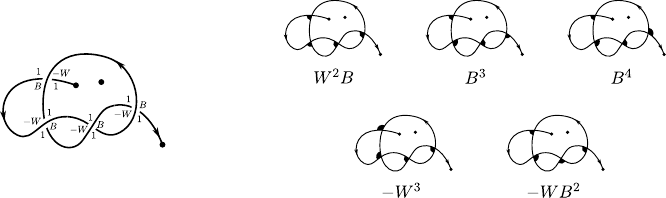}
\caption{The starred knotoid $K_1$ (left) and its states including weights (right).}
\label{fig:labeltwo}
\end{figure}

 
\begin{figure}[ht]
\centering
\includegraphics[width=\linewidth]{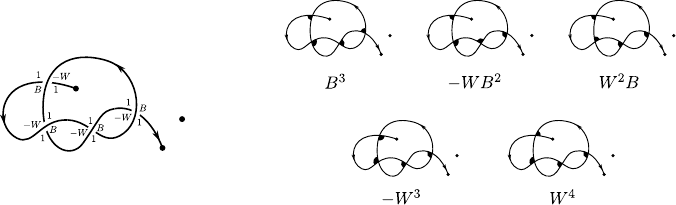}
\caption{The starred knotoid $K_2$ (left) and its states including weights (right).}
\label{fig:allstates2}
\end{figure}

\end{example}

The value of $\nabla_{L'}(W,B)$ evidently depends on the chosen way of starring $L$, and therefore cannot be used to construct an invariant of $L$ unless the construction of $L'$ is somehow canonical. We return to the question of starring $L$ canonically for knotoids in Section \ref{sec:knotoids}, and for linkoids in Section \ref{sec:linkoids}. For now we take for granted that a given starred linkoid is admissible, and derive an invariant from $\nabla_{L'}(W,B)$ as follows.

\begin{definition}\label{def:MAP}
Let $L$ be an admissible, connected (starred) linkoid diagram. The \textit{mock Alexander polynomial} of $L$, denoted $\nabla_L(W)$ by slight abuse of notation, is defined to be $\nabla_L(W,W^{-1})$ i.e.~the potential of $L$ after substituting $B=W^{-1}$.
\end{definition}

\begin{remark}
To explain the nomenclature `mock' Alexander polynomial we note that the Definition of $\nabla_L(W,B)$ given in \cite{MAP} is a generalization of the state sum formulation for the Alexander polynomial from \cite{kauffman1983formal}. Here it was shown that for a link diagram $L$ the obstruction is $\Omega(L)=2$, and if $L'$ is the starred link diagram obtained by starring two adjacent regions in a diagram of $L$ then $\nabla_{L'}(W)$ is equal to the Alexander polynomial of $L$ up to a change of variables.
\end{remark}

Since admissibility of a diagram is invariant under the Reidemeister moves (assuming all diagrams are connected, without loss of generality by Lemma \ref{lm:split}), the statement of the following theorem makes sense.

\begin{theorem}\cite{MAP}\label{thm:invariance}
The mock Alexander polynomial $\nabla_L(W)$ is invariant under surface isotopy and the Reidemeister moves, and hence is an invariant of admissible (starred) linkoids.
\end{theorem}
A proof of Theorem \ref{thm:invariance} is given in \cite{MAP}. 

\begin{example}\label{ex:noneqvknotoids}
Consider again the starred knotoids $K_1$ and $K_2$ from Example \ref{ex:one}. Then from their potentials we find

\begin{align*}
    &\nabla_{K_1}(W) = -W^3 + W -W^{-1} + W^{-3} + W^{-4} ,\\
    &\nabla_{K_2}(W) = W^4 - W^3 + W - W^{-1} + W^{-3} .
\end{align*}

Hence we find that $\nabla_{K_1}(W)\neq \nabla_{K_2}(W)$, so that $K_1$ and $K_2$ are inequivalent by Theorem \ref{thm:invariance}.

\end{example}


Notice that $\nabla_{K_2}(W,B)$ in Example \ref{ex:one} can be obtained from $\nabla_{K_1}(W,B)$ by making the substitution $(W,B)\leftrightarrow(-B,-W)$. This symmetry has been observed in many other examples of pairs of starred knotoid diagrams in which the regions containing either endpoint is starred. Accordingly, the following conjecture was given in \cite{MAP}, which remains open at the time of writing:

\begin{conjecture}\label{conj:conj1}\cite{MAP}
Let $K$ be a knotoid diagram, and $K_1$, $K_2$ be two starred knotoid diagrams obtained from $K$ by placing a star in the region of $K$ incident to the tail and the head of $K$, respectively.
Then
$$\nabla_{K_{1}}(W, B) = \nabla_{K_{2}} (-B, -W).$$
\end{conjecture}

We prove this conjecture for the case of the mock Alexander polynomial, i.e.~when $B=W^{-1}$, in section \ref{sec:knotoids}.

To compute the potential of an admissible diagram, a priori one must enumerate its states, and there is no clear way of doing so systematically. To ease example computations from here on out we note an alternative method of computing $\nabla_L(W,B)$, namely as the permanent of a matrix. 

\begin{definition}\normalfont
    
Let $L$ be an admissible starred linkoid diagram with $n \geq 1$ crossings and regions without stars which are enumerated by $\{1, 2, ..., n\}$. The \textit{potential matrix} $M_L$ of $L$ is the $n \times n$ matrix whose $ij^{th}$ entry is the local weight placed in the $j$-th region by the $i$-th crossing in accordance with Figure \ref{fig:weights}, or is zero if region $j$ is not incident to crossing $i$. If the $i$-th crossing has multiple adjacent quandrants lying in the $j$-th region, then the $ij$-th entry of $M_L$ is the sum of corresponding local weights.
\end{definition}

\begin{proposition}\label{prop:perm}\cite{MAP}
Let $L$ be a starred link or linkoid diagram in a surface, with $n \geq 1$ crossings and regions.
The potential $\nabla_{L}(W,B)$ of $L$ is equal to the permanent of the potential matrix $M_L$ of $L$. That is, $\nabla_{L}(W)=  \text{Perm}(M_{L})$.
\end{proposition}

\begin{corollary}
The mock Alexander polynomial can also be obtained as the permanent of a matrix, namely $\nabla_L(W) = \text{Perm} \left( M_L \vert_{B=W^{-1}} \right) $.
\end{corollary}

\begin{example}
Let $K$ be the knotoid diagram depicted in Figure \ref{fig:combex}, so that $K_1$ and $K_2$ from Example \ref{ex:one} are obtained from starring the regions $A$ and $E$ respectively. Then the matrix $M$ containing the local weights of $K$ is given in Figure \ref{fig:combex}, and $M_{K_1}$ and $M_{K_2}$ are obtained from $M$ by deleting the first and the fifth column, respectively. We compute that $\text{Perm}(M_{K_1})=-W^3 + W^2B -WB^2 + B^3 + B^4$ and $\text{Perm}(M_{K_2})=W^4 - W^3 + W^2B - WB^2 + B^3$, and these indeed agree with the findings of Example \ref{ex:one}.

\end{example}

\begin{figure}[ht]
\centering
\includegraphics[width=.7\textwidth]{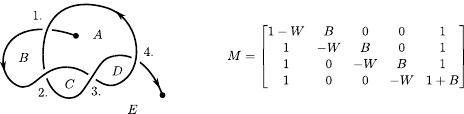}
\caption{The weight matrix of $K$, from which the potential matrices of $K_1$ and $K_2$ are obtained by deleting one column.}
\label{fig:combex}
\end{figure}

\subsection{The mock Alexander polynomial for generalized knotoids}

To study linkoids in Section \ref{sec:linkoids} we will use generalized knotoids and their mock Alexander polynomials. For the purpose of defining a state sum for (starred) generalized knotoid diagrams we interpret a star in a starred diagram as a degree zero or four node of a generalized knotoid, recalling Remark \ref{rk:star_gen}.
Recalling Example \ref{ex:planar_linkoids} we can also interpret the point at infinity of a diagram on the plane as a degree zero node in a generalized knotoid, so that it suffices to consider the mock Alexander polynomial of generalized knotoids on the sphere since we can add a degree zero node at $\infty\in S^2$ to derive a mock Alexander polynomial for planar diagrams.


\begin{definition}
Let $G$ be a connected generalized knotoid diagram, $n$ the number of crossings in $G$, and $f_d$ the number of regions in $G$ that are homeomorphic to a disk $D^2$. Note here that since $G$ is connected and spherical, a region is homeomorphic to a disk if and only if it contains no degree zero nodes. We say $G$ is \textit{admissible} if $n=f_d$. The \textit{generalized obstruction} of $G$ is defined by $\Omega_g(G)=n-f_d$. A \textit{state} of a generalized knotoid is a bijection assigning an adjacent disk-shaped region to each crossing of $G$. The \textit{potential} and \textit{mock Alexander polynomial} of $G$ are defined by the expressions in Definitions \ref{def:potential} and \ref{def:MAP}, as before.
\end{definition}

As for starred linkoids, the mock Alexander polynomial of a generalized knotoid can be computed as the permanent of the potential matrix that is the matrix of weights given by crossings and regions without degree zero nodes.

\begin{example}
Let $G$ be the generalized knotoid diagram depicted in Figure \ref{fig:generalized_polynomial}. Then $G$ is admissible as $n=f_d=3$, and has eight states. For the potential and mock Alexander polynomial of $G$ we find
\begin{align*}
    & \nabla_{G}(W,B) = W^2 - WB + 2W - 2B + 2,\\
    & \nabla_{G}(W) = W^2 + 2W - 2W^{-1} + 1.\\
\end{align*}
\end{example}

\begin{figure}[ht]
    \centering
    \includegraphics[width=0.38\linewidth]{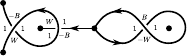}
    \caption{An admissible generalized knotoid. Crossings labelled with local weights.}
    \label{fig:generalized_polynomial}
\end{figure}

Our reason for introducing generalized knotoids is on the one hand that planar and spherical linkoids form sub-classes of generalized knotoids, while simultaneously generalized knotoids can be easily augmented to control $\Omega_g$ by the following Lemma:

\begin{lemma}\label{lm:edges}
Let $G$ be a connected generalized knotoid diagram and let $\tilde{G}$ be obtained from $G$ by adding an edge between two nodes of nonzero degree. Then $\Omega_g(\tilde{G})=\Omega_g(G)-1$.
\end{lemma}

\begin{proof}
Let $G$ have $n$ crossings and $f$ faces, and $v$ and $w$ be two vertices of $G$ of degree greater than zero. Pick an arc embedded in $\mathbb{R}^2$ that connects $v$ and $w$, and suppose there are $k \geq 0$ arcs that the connecting arc crosses during the connection. It is clear that the number of crossings in the generalized knotoid diagram $\tilde{G}$ obtained after connecting $v$ and $w$ with the chosen arc is $k$ more than the number of crossings in $G$. Each crossing is counted as a degree four vertex. Then, the number of edges of $\tilde{G}$ is $k+1$ more than the number of edges in $G$. It follows by Euler's formula for planar graphs that the number of regions in $\tilde{G}$ is $k+1$ more than the number of regions in $G$. 
Thus we find
\[
    \Omega_g(\tilde{G})= (n+k)- (f+k+1)= \Omega_g(G)-1.
\]
\end{proof}

Interpreting planar linkoids as starred linkoids and therefore in turn as generalized knotoids as in Section \ref{subsec:generalized}, we can rephrase Proposition \ref{prop:obstruction} in terms of $\Omega_g$. In Section \ref{sec:linkoids} we will interpret all linkoids as generalized knotoid diagrams, and therefore use this generalized obstruction formula.

\begin{proposition}\label{prop:obstruction_g}
Let $L$ be a connected $(\kappa,\ell)$-linkoid diagram. Then
\[
\Omega_g(L) = 
\begin{cases}
\kappa - 2 & \qquad \text{if $L$ is spherical,}\\
\kappa - 1 & \qquad \text{if $L$ is planar.}
\end{cases}
\]
\end{proposition}

\begin{proof}
This follows from Proposition \ref{prop:obstruction} since $\Omega_g(L)=\Omega(L)$ if $L$ is spherical, and since any linkoid diagram in $\mathbb{R}^2$, seen as a generalized knotoid, has one fewer disk-shaped region than its corresponding diagram on $S^2$ due to the degree zero node at infinity.
\end{proof}

\subsection{The skein relation}

Aside from the potential matrix, another useful tool for computing mock Alexander polynomials is the \textit{Conway skein relation}. It works by relating the polynomials of three links that differ from each other only at a single crossing. This relation for the classical Alexander polynomial was discovered in \cite{conway1970enumeration} for classical knots and links. For the mock Alexander polynomial the skein relation was also shown to hold in \cite{MAP}, but only at selected crossings. We recall the relevant results for linkoids here.

\begin{definition}
Let $L$ be a linkoid diagram and $c$ a crossing of $L$. We say that $c$ is a \textit{separating} crossing if one of the two possible (unoriented) smoothings of $c$ results in a split diagram. (Recall Definition \ref{def:split}.)
\end{definition}

\begin{theorem}\label{thm:skein}\cite{MAP}
Let $L_+$, $L_-$, and $L_0$ be three starred linkoid diagrams that differ from each other only in a small disk as is depicted in Figure \ref{fig:skein_relation}. Let $c$ be the crossing in $L_+$ (or in $L_-$) contained in this disk. Then if $c$ is not separating, the \textit{skein relation}
\begin{equation}\label{eq:skein}
    \nabla_{L_+}(W) - \nabla_{L_-}(W) = (W-W^{-1}) \nabla_{L_0}(W)
\end{equation}
holds.
\end{theorem}

\begin{figure}[ht]
    \centering
    \includegraphics[width=0.6\linewidth]{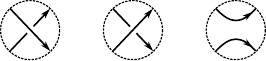}
    \caption{$L_+$, $L_-$, and $L_0$ respectively, within the disk on which they differ.}
    \label{fig:skein_relation}
\end{figure}

The skein relation can be used to reduce the complexity of diagrams by replacing a crossing with a linear combination of two, potentially simpler, diagrams. However unlike for classical links, for linkoids the skein relation cannot be used to reduce any diagram to a linear combination of trivial diagrams. Therefore the relation can be used to aid the computation of $\nabla_L(W)$, but it does not determine the invariant completely. We can also ask how far one can get using only the skein relation. This naturally leads to the following definition of \textit{skein modules}.

\begin{definition}
We call a triple $(L_+,L_-,L_0)$ of linkoids differing only on a disk as in Figure \ref{fig:skein_relation} a \textit{Conway triple}. Let $V_L$ be the free $\mathbb{Z}[W^{\pm1}]$-module generated by equivalence classes of spherical linkoids. Similarly let $V_U$ be the free $\mathbb{Z}[W^{\pm1}]$-module generated by uni-linkoids. The \textit{linkoid skein module} $\mathcal{S}_L$ of $S^2$ is defined to be the quotient of $V_L$ by the submodule generated by all vectors of the form $L_+ - L_- - (W-W^{-1}) L_0$ where $(L_+,L_-,L_0)$ is a Conway triple of linkoids. Similarly the \textit{uni-linkoid skein module} $\mathcal{S}_U$ is defined to be $V_L$ modulo the submodule generated by the same vectors, ranging over all Conway triples of uni-linkoids.
\end{definition}

Using that the Conway skein module of link diagrams in the annulus is known, the structure of $\mathcal{S}_U$ was derived in \cite{turaev2012knotoids}:

\begin{theorem}\label{thm:skein_module}\cite{turaev2012knotoids}
The uni-linkoid skein module $\mathcal{S}_U$ of $S^2$ is a free $\mathbb{Z}[W^{\pm1}]$-module with generating set $G_n$ for $n\geq 0$. Here $G_0$ is the trivial knotoid diagram, and $G_n$ for $n>0$ is the uni-linkoid diagram given by a trivial knotoidal component and a single closed component that winds around the tail of the knotoidal component $n$ times creating $n$ crossings, and then closes via an over-going arc that creates $n-1$ more crossings.
\end{theorem}

An example of a generator of $\mathcal{S}_U$ is depicted in Figure \ref{fig:generator}.

\begin{figure}[ht]
    \centering
    \includegraphics[width=0.23\linewidth]{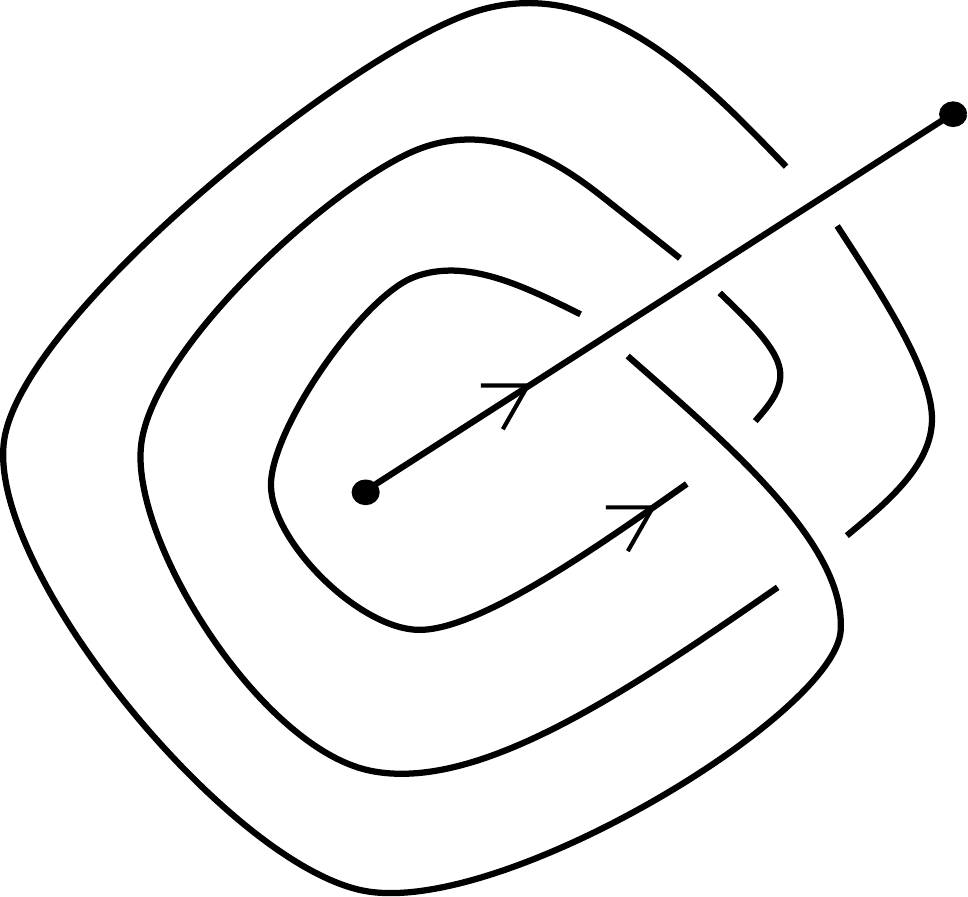}
    \caption{An example of a generator $G_n$ of $\mathcal{S}_U$, for $n=3$.}
    \label{fig:generator}
\end{figure}

\section{The mock Alexander polynomial for uni-linkoids}\label{sec:knotoids}

\subsection{Planar knotoids}

Let $K$ be a planar uni-linkoid, seen as a generalized knotoid. Then by Proposition \ref{prop:obstruction_g}, $\Omega_g(K)=0$ and any connected diagram of $K$ is admissible. As such we can compute $\nabla_K(W)$ immediately, without altering $K$, to obtain an invariant of planar knotoids and uni-linkoids:

\begin{lemma}
The mock Alexander polynomial $\nabla_K(W)$ constitutes an invariant of planar uni-linkoids, seen as starred spherical uni-linkoids with a star placed in the exterior region.
\end{lemma}

The importance of this is the following: While spherical uni-linkoids and particularly knotoids are relatively well-understood, planar knotoids have proven to be more elusive. In \cite{goundaroulis2019systematic} a tabulation of spherical and planar knotoids is given, and while spherical knotoids could be classified up to crossing number six, for planar knotoids even the classification for crossing number $5$ remains challenging \cite{moltmaker2023new}. The reason for this is that the construction of most invariants is automatically invariant under spherical moves, meaning they factor through the surjection from planar to spherical knotoids (recall Remark \ref{rk:planar-spherical}). The construction of invariants that respect the Reidemeister moves but are not invariant under spherical moves is therefore of active research interest. Below we give an example showing that $\nabla_K(W)$ is indeed such an invariant.

\begin{example}
Let $K_1$ and $K_2$ be the knotoid diagrams depicted in the left- and right-hand sides of Figure \ref{fig:planar_example} respectively. The diagram $K_2$ is obtained from $K_1$ by a single spherical move, namely that which moves the point at infinity to the bigon region of $K_1$. We find that $K_1$ has three states while $K_2$ has five, and that
\begin{align*}
    &\nabla_{K_1}(W) = 1\cdot B + 1\cdot (-W) + (-W)^2 = W^2-W+W^{-1},\\
    &\nabla_{K_2}(W) = 1\cdot 1 + 1\cdot 1 + 1\cdot B + (-W)\cdot B + (-W)\cdot 1 = 1-W+W^{-1}.
\end{align*}
Hence $K_1$ and $K_2$ are inequivalent planar knotoids and this is detected by $\nabla$, even though they are equivalent as spherical knotoids.
\end{example}

\begin{figure}[ht]
    \centering
    \includegraphics[width=.6\linewidth]{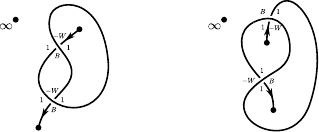}
    \caption{Two inequivalent planar knotoids. Crossings are decorated with local weights.}
    \label{fig:planar_example}
\end{figure}

\subsection{Canonical starrings on the sphere}

While planar uni-linkoids are admissible so that the mock Alexander polynomial constitutes a canonical invariant, the same is not true for spherical uni-linkoids which have $\Omega_g(L)=-1$ by Proposition \ref{prop:obstruction_g}. As discussed in Section \ref{sec:MAP}, spherical uni-linkoids can be made admissible by adding a single degree zero node to them. However the resulting mock Alexander polynomial may depend on the choice of region in which this degree zero node is placed. For a uni-linkoid there are two natural choices of where to place this node, namely the regions containing the tail and head of the knotoidal component. As we shall prove in this section, these two choices yield equal mock Alexander polynomials up to a change of variables $W\leftrightarrow -W^{-1}$. Before proving this in Theorem \ref{thm:knotoids} we first treat a few more preliminaries.

\begin{lemma}\cite{MAP}\label{lm:rev}
Let $K$ be a starred uni-linkoid diagram, and $-K$ be its reverse obtained by reversing the orientation on each component of $K$. Then
\[
    \nabla_{K}(W) = \nabla_{-K} (-W^{-1}).
\]
\end{lemma}

\begin{proof}
This follows from the assignment of local weights in the definition of $\nabla_K$, see Figure \ref{fig:weights}.
\end{proof}

\begin{definition}
A crossing in a linkoid diagram $L$ is said to be \textit{removable} if it is of the form shown in Figure \ref{fig:removable_crossing}, i.e. connects on one side to a $(1,1)$-tangle $T$. (Here by an $(n,m)$-tangle we mean a tangle with $n$ ingoing ends and $m$ outgoing ends.)
\end{definition}

\begin{figure}[ht]
    \centering
    \includegraphics[width=0.2\linewidth]{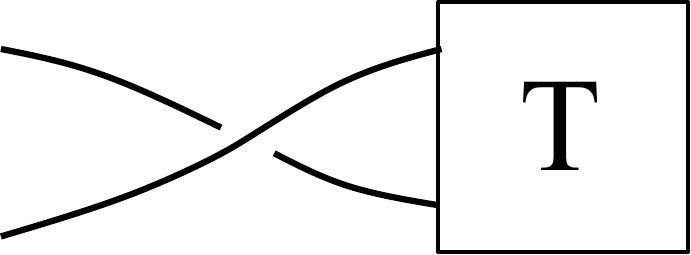}
    \caption{A removable crossing. Here $T$ represents a tangle with one in-going and one out-going end.}
    \label{fig:removable_crossing}
\end{figure}

Clearly removable crossing can be removed using Reidemeister moves by `flipping' $T$ over. For linkoids with several knotoidal components there are examples of separating crossings that are not removable: consider for example the unique $(2,0)$-linkoid diagram with exactly one crossing. We show this is not the case for uni-linkoids:

\begin{lemma}\label{lm:separating}
Let $K$ be a uni-linkoid, and let $c$ be a separating crossing in $L$. Then $c$ is removable.
\end{lemma}

\begin{proof}
Since $c$ is separating it must be of the form depicted in Figure \ref{fig:separating}, with two of its adjacent ends connecting to a $(1,1)$-tangle $T_1$ and with the other ends connecting to another $(1,1)$-tangle $T_2$. If both endpoint of $K$ lie in either $T_1$ or $T_2$ then $c$ is removable. Suppose for a contradiction that both $T_1$ and $T_2$ contain exactly one endpoint of $K$, as in Figure \ref{fig:separating}. Then following the knotoidal component of $K$, we begin in $T_1$, move through one of the arcs constituting $c$, and end in $T_2$. After moving through one of the arcs of $c$, the component cannot return to $T_1$ as it must end in $T_2$ and there is only one arc of $c$ that the component has not moved through yet. Therefore this other arc of $c$ must lie on a closed component of $K$. But now this closed component moves between $T_1$ and $T_2$ exactly once, i.e.~an odd number of times, contradicting that the component is closed. Therefore we arrive at a contradiction.
\end{proof}

\begin{figure}[ht]
    \centering
    \includegraphics[width=0.35\linewidth]{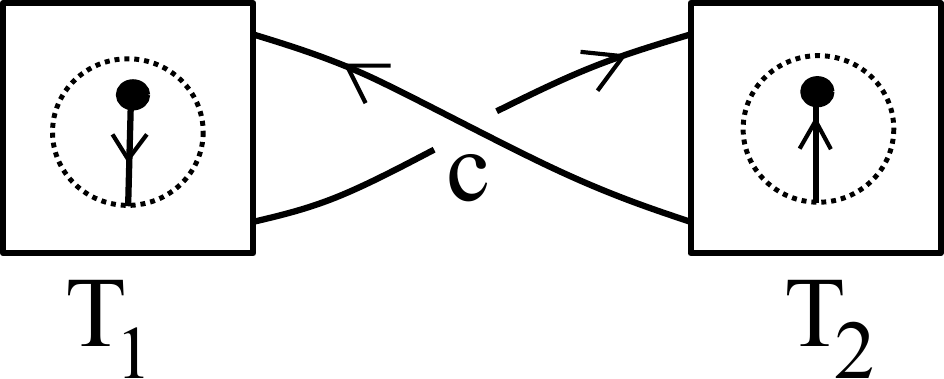}
    \caption{Proof of Lemma \ref{lm:separating}: a separating crossing that is not removable.}
    \label{fig:separating}
\end{figure}

\begin{theorem}\label{thm:knotoids}
Let $K$ be an oriented spherical uni-linkoid diagram. Let $D^t$ and $D^h$ be the starred knotoid diagrams given by endowing the tail and head regions of $K$ with stars, respectively. Then
\[
    \nabla_{D^t} (W) = \nabla_{D^h}(-W^{-1}).
\]
\end{theorem}


\begin{proof} 
Let $\delta_K(W) = \nabla_{D^t}(W)-\nabla_{D^h}(-W^{-1})$. Then we wish to show that $\delta_K(W)=0$. Since $\delta_K$ is a linear combination of mock Alexander polynomials we know that $\delta_K$ satisfies the skein relation \eqref{eq:skein} at non-separating crossings by Theorem \ref{thm:skein}. By Theorem \ref{thm:skein_module} any uni-linkoid can be reduced to a $\mathbb{Z}[W^{\pm1}]$-linear combination of generators $G_n$ using the skein relation. 

\textbf{Claim:} Any uni-linkoid can be reduced to a linear combination of the generators using only applications of the skein relation at non-separating crossings, as well as the Reidemeister moves. 

\textit{Proof of claim:} Take a sequence of applications of the skein relation that reduces $K$ to a linear combination of generators. Take the first skein relation in this sequence that involves a separating crossing $c$. Then by Lemma \ref{lm:separating} $c$ is removable, and adjacent to a $(1,1)$-tangle $T$ as depicted in Figure \ref{fig:removable_crossing}. The action of the skein relation on $c$ therefore consists of replacing it by a linear combination of a diagram in which $c$ has been flipped, and a diagram in which $c$ has been smoothed to split $T$ off as a link disjoint from the rest of the diagram. To further reduce $K$ one then applies the skein relation to reduce this link to the empty diagram. To get rid of the application of the skein relation at $c$, we instead use the Reidemeister moves to remove $c$ and reduce $T$ to the trivial tangle without splitting it off as a link. For this we use the standard argument showing that the Alexander skein module for links is trivial, i.e.~generated by the empty diagram: $T$ is said to be \textit{descending} if, when running through the components of $T$, each crossing of a component with itself is encountered first as an over-crossing, and for all pairs of distinct components $C_1,C_2$ if $C_1$ makes an over-crossing with $C_2$ then all other crossings that $C_1$ makes with $C_2$ are also over-crossings. Let $n_c$ be the crossing number of $T$, let $D$ be the set of crossings in $T$ that need to be flipped in order to make $T$ descending, and let $n_d=\lvert D\rvert$. Select a crossing $d$ in $D$. If $d$ is separating, then it is removable by Lemma \ref{lm:split} and we remove it, lowering $n_c$ by one. Otherwise apply the skein relation to $d$, replacing $T$ with the sum of a diagram in which $d$ has been flipped and a diagram from which $d$ has been removed. In the former diagram, $n_c$ has remained constant and $n_d$ has been lowered by one. In the latter, both $n_c$ and $n_d$ have been lowered by one. It is clear that if $n_c=0$ or $n_d=0$ then $T$ is trivial. Therefore by induction $T$ can be reduced to a linear combination of trivial diagrams using only the skein relation on non-separating crossings and the Reidemeister moves. This proves the claim.

By the claim, as well as Theorem \ref{thm:skein} and invariance of $\nabla_K(W)$, we conclude that $\delta_K$ is a $\mathbb{Z}[W^{\pm1}]$-linear combination of $\delta_{G_n}$'s. Therefore it suffices to prove the statement for $G_n$ for all $n\in \mathbb{Z}_{\geq 0}$.


To see that $\delta_{G_n}(W)=0$ it suffices to note that $G_n$ looks identical when viewed from the tail as when viewed from the head, but with the orientations of both of its components swapped. This can be seen by pushing all the arcs encircling the tail of $G_n$ past $\infty\in S^2$, to encircle the head instead. See Figure \ref{fig:generator_example} for an example of this procedure. 

\begin{figure}[ht]
    \centering
    \includegraphics[width=.9\linewidth]{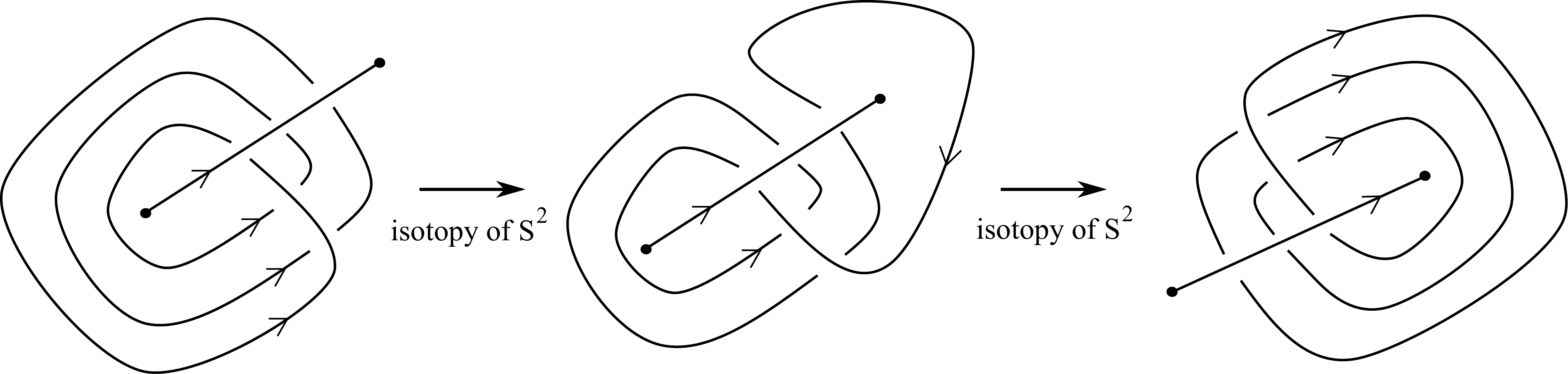}
    \caption{Relating the generator $G_3$ to its reverse on $S^2$.}
    \label{fig:generator_example}
\end{figure}

From this we infer that $\nabla_{G_n^t} (W)=\nabla_{(-G_n)^t} (W)$, where $-G_{n}$ denotes the reverse of $G_{n}$. Equivalently, $-(G_n^t)$ is isotopic to $G_n^h$. Therefore we have:
\begin{align*}
    \delta_{G_n}(W) &= \nabla_{G_n^t}(W) - \nabla_{G_n^h}(-W^{-1})\\
    &= \nabla_{G_n^t}(W) - \nabla_{-(G_n^t)}(-W^{-1})\\
    &= \nabla_{G_n^t}(W) - \nabla_{G_n^t}(W)\\
    &=0
\end{align*}
as required. Here the penultimate equality follows from Lemma \ref{lm:rev}.


\end{proof}

Theorem \ref{thm:knotoids} can be rephrased as follows.

\begin{corollary}
Let $K$, $D^t$, and $D^h$ be as before. Then
\[
    \nabla_{D^t} (W) = \nabla_{-(D^h)}(W).
\]
\end{corollary}

\begin{remark}
Theorem \ref{thm:knotoids} implies that, up to a change of variables, there is a single canonical mock Alexander polynomial for uni-linkoids on the sphere. We discuss the problem of associating a canonical polynomial to linkoids with several knotoidal components in Section \ref{sec:linkoids}.
\end{remark}

As another corollary to Theorem \ref{thm:knotoids} we obtain a formula for the mock Alexander polynomial of the `\textit{virtual closure}' of a knotoid. Let $K$ be a knotoid diagram in $S^2$. The virtual closure $v(K)$ of $K$ is obtained by adding an arc between the endpoint of $K$ and declaring every other intersection point of this arc with $K$ to be a virtual crossing. The result is a virtual knot in $S^2$ which, by the equivalent formulation of virtual knots as knot diagrams on surfaces, can also be seen as a knot on the torus. This is done by adding a handle to $S^2$ between the regions containing the endpoints of $K$ and letting the added arc between the endpoints run along this handle. Figure \ref{fig:virtualclosure} depicts the virtual closure of a knotoid diagram and its torus representation.

 \begin{figure}[ht]
\centering
\includegraphics[width=.85\textwidth]{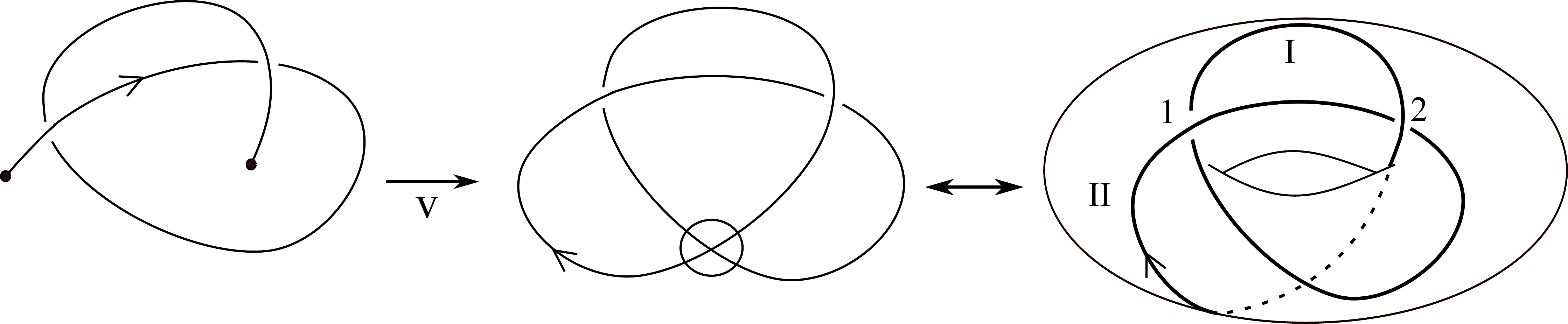}
\caption{The virtual closure of a knotoid \cite{MAP}.}
\label{fig:virtualclosure}
\end{figure}

The virtual closure of any knotoid diagram is an admissible diagram in the torus, as long as its endpoints lie in distinct regions. An example from \cite{MAP} gives the potential matrix of the virtual knot $v(K)$ in Figure \ref{fig:virtualclosure} as 

\[
    M_{v(K)} = 
    \begin{bmatrix}
    -W^{-1} & W+2\\
    W& 2-W^{-1} \\
    \end{bmatrix}.
\]

By calculating the permanent of this matrix they find that

\[
    \nabla_{v(K)} (W) = W^2 + W^{-2}+ 2( W-W^{-1}).
\]

It can be checked that in this case, the mock Alexander polynomial of $v(K)$ is equal to $\nabla_{K^{t}}(W) + \nabla _{K^{h}} (W)$. Indeed, this was proven to hold generally in \cite{MAP}.

\begin{lemma}\label{lm:virtual_closure}
\cite{MAP}
Let $K$ be a knotoid diagram in $S^2$, and $v(K)$ denote the torus representation of its virtual closure. Then $$\nabla_{v(K)} (W)= \nabla _{K^{t}} (W) + \nabla_{K^{h}}(W).$$
\end{lemma}

 
Combining Theorem \ref{thm:knotoids} and Lemma \ref{lm:virtual_closure}, we obtain the following identity.

\begin{corollary}
    
Let $K$ be a knotoid in $S^2$ and $v(K)$ denote its virtual closure.

$$\nabla_{v(K)} (W) = \nabla_{K^{t}} (W) + \nabla_{K^{t}} (-W^{-1}).$$

\end{corollary}

\section{The mock Alexander polynomial for linkoids}\label{sec:linkoids}

As was mentioned in Section \ref{sec:MAP}, while any linkoid diagram can obviously be made admissible by replacing it with an appropriate starred linkoid diagram, this replacement is not necessarily made canonically. In fact the starred linkoid diagram is usually not an invariant of the original linkoid. For example the result of starring some crossings of a diagram is clearly not invariant under $R3$ moves in general. 

In Section \ref{sec:knotoids} we considered the canonical starrings for spherical uni-linkoids and showed that as far as $\nabla_K(W)$ is concerned these starrings are essentially equivalent. In this section we consider the problem of canonically associating a mock Alexander polynomial of generalized knotoid diagrams to a linkoid with multiple knotoidal components. Seeing starred linkoids as generalized knotoids, this in some sense extends the application of Theorem \ref{thm:knotoids} to the construction of uni-linkoid invariants.

Given a spherical or planar $(\kappa,\ell)$-linkoid diagram $L$ with $\kappa>2$, we know that $\Omega_g(L)\geq 0$ by Proposition \ref{prop:obstruction_g}. As such we need to find a way of replacing $L$ by a generalized knotoid diagram $G(L)$, such that $\Omega_g(G(L)) < \Omega_g(L)$ and $G(L)$ is an invariant of $L$. To this end we turn to Lemma \ref{lm:edges}, which tell us that we can add edges between the endpoints of $L$ to reduce $\Omega_g(L)$ by one per added edge. Now the question is how to add these edges in a canonical and invariant way.

\subsection{Closures of uni-linkoids}

For uni-linkoids, the basic well-defined edge additions are its over-closure and under-closure, which are given by adding an edge between the uni-linkoid's endpoints that makes only over-crossings or only under-crossings, respectively. The resulting diagram is independent of the choice of embedding for this edge, since any two choices are related by Reidemeister moves. For $(\kappa,\ell)$-linkoids with $\kappa>2$ this is no longer the case as the endpoints of other components can obstruct the movement of these closing arcs. To remedy this, below we give a new definition of closure for a knotoidal component, which does not depend on any choices.

\begin{definition}\label{def:under-closures}
Let $L$ be a linkoid diagram, and choose a subset $S$ of the set of knotoidal components of $L$. We define the \textit{shadow under-closure} of $S$ to be the diagram given as follows: Let $C$ be a component in $S$. Then we add an edge to $L$ that runs parallel to $C$, always staying on the same side inside a small neighbourhood of $C$, and that makes only under-crossings with the original components of $L$. Doing so for all components in $S$ may create crossings between the added edges. Such crossings will have a corresponding crossing in the original diagram for $L$, and in the shadow closure we take the crossing between the added edges to be identical to its associated crossing in $L$. Finally, we give the added edges an orientation. This orientation can be chosen either parallel or anti-parallel to the original components of $L$. In the anti-parallel case the knotoidal components are closed into loop components, while in the parallel case they are generalized knotoid components with two degree two nodes where the orientation changes. We denote the parallel and anti-parallel shadow under closures of $S$ in $L$ by $u_{s,p}(L;S)$ and $u_{s,a}(L;S)$ respectively.


Another way to form a well-defined under-closure is to define the crossings between added edges to be \textit{opposite} to their associated crossings in the original linkoid. The result is another well-defined closure of $S$ in $L$ which we call the \textit{mirror under-closure}. Its parallel and anti-parallel versions are denoted $u_{m,p}(L;S)$ and $u_{m,a}(L;S)$ respectively.

\begin{figure}[ht]
    \centering
    \includegraphics[width=.7\linewidth]{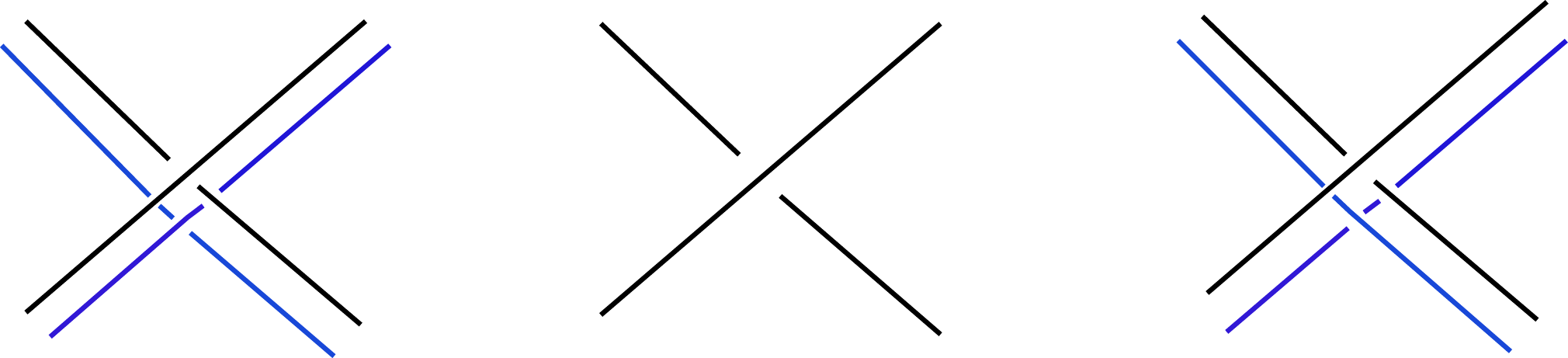}
    \caption{The shadow under-closure (left) and mirror under-closure (right) at a crossing with both components in $S$.}
    \label{fig:closures}
\end{figure}

\end{definition}

Over-closures are defined analogously to Definition \ref{def:under-closures}:

\begin{definition}\label{def:over-closures}
Let $L$ be a linkoid, and choose a subset $S$ of the set of knotoidal components of $L$. We define over-closures of $S$ by adding arcs running parallel to the components in $S$ and making only over-crossings with $L$. 
We define the \textit{(parallel and anti-parallel) shadow over-closures} $o_{s,p}(L;S)$ and $o_{s,a}(L;S)$ by choosing crossings between the over-going arcs to mimic their associated crossing in $L$, and define the \textit{(parallel and anti-parallel) mirror over-closures} $o_{m,p}(L;S)$ and $o_{m,a}(L;S)$ by choosing them to be opposite. Here the distinction between parallel and anti-parallel closures lies in the orientation of the added edges, analogously to Definition \ref{def:under-closures}.
\end{definition}

So we have a total of eight different canonical closures, with a choice between over/under, shadow/mirror, and parallel/anti-parallel. 

\begin{lemma}
For $L$ a linkoid and $S$ a subset of its knotoidal components, its shadow- and mirror-closures are well-defined.
\end{lemma}

\begin{proof}
The arcs added by each of these closures can always be thought of as lying on a plane separate from that of the original linkoids, except where the arcs connect to the linkoid's endpoints. Since the closure arcs constitute parallel copies of the linkoid's components, a Reidemeister move in the linkoid is translated to two copies of that same move in the closure (identical copies for the shadow closures, and mirror copies for the mirror closures). Since these copies lie on separate planes they don't interfere with each other, and invariance under all three Reidemeister moves follows. See Figure \ref{fig:R3_closure} for an example of this for an $R3$ move.
\begin{figure}[ht]
    \centering
    \includegraphics[width=0.6\linewidth]{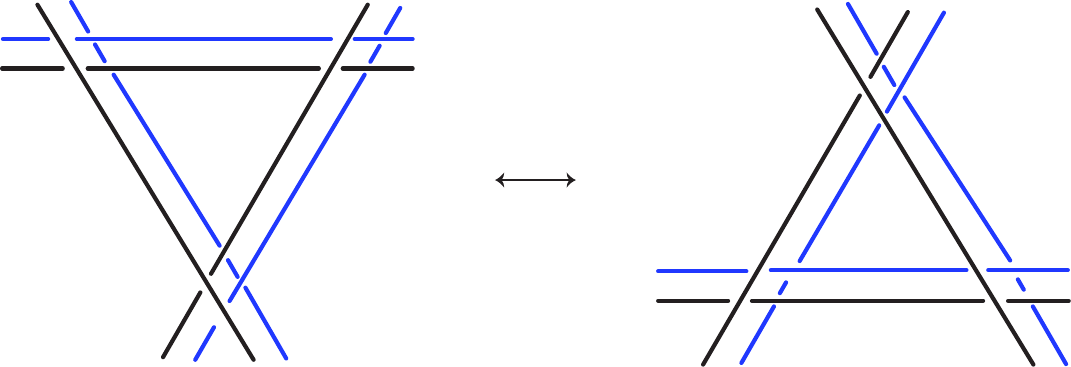}
    \caption{Invariance of $u_{m,p}(L;S)$ and $u_{m,a}(L;S)$ under $R3$.}
    \label{fig:R3_closure}
\end{figure}
\end{proof}

\begin{example}
Figure \ref{fig:ex_closures} shows a linkoid $L$ with knotoidal components $C_1,C_2$, and depicts the parallel shadow under-closure $u_{s,p}(L;\{C_1\})$ of $C_1$, as well as its anti-parallel mirror over-closure $o_{m,a}(L;\{C_1\})$.
\begin{figure}[ht]
    \centering
    \includegraphics[width=\linewidth]{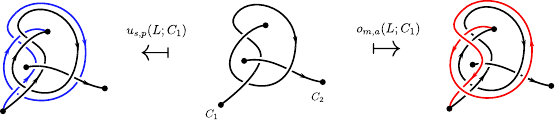}
    \caption{Parallel shadow under-closure and anti-parallel mirror over-closure of an example linkoid.}
    \label{fig:ex_closures}
\end{figure}
\end{example}

The shadow and mirror closures are in general inequivalent and have varying distinguishing power. For an example of this, see Figure \ref{fig:closed_crossing}.

\begin{figure}[ht]
    \centering
    \includegraphics[width=\linewidth]{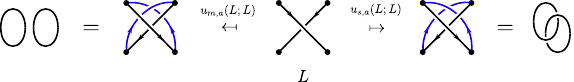}
    \caption{The anti-parallel mirror and shadow  under-closures of a single-crossing $(2,0)$-linkoid $L$. Note only the shadow closure detects non-triviality of $L$.}
    \label{fig:closed_crossing}
\end{figure}

\begin{remark}
Note that the shadow closure doubles the contribution of a crossing to the linking number, while the mirror closure cancels it. This may weaken its distinguishing power. In fact if one applies the mirror closure to every component of a linkoid the result is clearly a ribbon link, as in Figure \ref{fig:closed_crossing}, but this map from linkoids to ribbon links is easily seen not to be surjective.
\end{remark}

While it is interesting to compare the shadow and mirror closures, the shadow closures generally have more distinguishing power as a consequence to the following lemma. As such we will mainly consider the shadow closure in subsequent sections.

\begin{lemma}
The anti-parallel shadow under-closure detects under-type forbidden moves when both of the strands involved in the move are closed, while the mirror under-closure does not. Neither under-closure detects under-type forbidden moves where only the endpoint-strand has been closed. Analogous statements hold for the over-closures and over-type forbidden moves.
\end{lemma}

\begin{proof}
This follows immediately from the definitions of $u_{m,a}(L;S)$ and $u_{s,a}(L;S)$.
\end{proof}

In contrast, the parallel over- and under-closures create degree two generalized knotoid nodes, and are therefore invariant under none of the forbidden moves. However the mock Alexander polynomial of these generalized knotoids will be invariant under forbidden moves in some cases: specifically, the mock Alexander polynomial of a parallel under-closure is invariant under over-type forbidden moves in which only endpoint-strand has been closed, and similarly for parallel over-closures.

In light of these remarks, in the next subsection we only make use of the parallel shadow over- and under-closures. It would be interesting to determine if the shadow closure is always stronger in distinguishing power than the mirror closure.

\subsection{Canonical polynomials}

Our approach to constructing canonical mock Alexander polynomials for linkoids is now as follows: forming shadow closures until a linkoid diagram becomes admissible relies only on a choice of finitely many knotoidal components of a linkoid $L$. Each choice has only finitely many options, with the number of choices depending on $\kappa$. So when there are several options for a choice, we are at liberty to consider the formal sum of generalized knotoids obtained from all possible options. We can then evaluate the mock Alexander polynomial state sum on this linear combination of generalized knotoid diagrams, to obtain a polynomial for $L$ that does not depend on any choices. The result is thus a canonical polynomial for the linkoid. We first illustrate this principle using the shadow under-closure below.

\begin{definition}\label{def:linkoid_MAP}
Let $L$ be a spherical $(\kappa,\ell)$-linkoid. The \textit{under-closure mock Alexander polynomial} $\nabla^u_L(W)$ of $L$ is defined as follows, depending on $\kappa$:
\begin{itemize}
    \item If $\kappa=0$, then $L$ is a link and $\nabla^u_L(W)$ is defined to be the Alexander polynomial of $L$, i.e. $\nabla_{\tilde{L}}(W)$ where $\tilde{L}$ is the starred link obtained from $L$ by starring two adjacent regions.
    \item If $\kappa=1$, $L$ is a uni-linkoid and $\Omega_g(L)=1$. In this case it suffices to star one of the endpoint regions of $L$; we choose to star the tail region without loss of generality up to a change of variables, by Theorem \ref{thm:knotoids}. Calling the resulting starred linkoid $\tilde{L}$, we define $\nabla^u_L(W)=\nabla_{\tilde{L}}(W)$.

     \item If $\kappa=2$, then $L$ is admissible, and so $\nabla^{u}_{L}$ is well-defined and can be computed directly.
    \item If $\kappa>2$, enumerate the knotoidal components of $L$ as $\mathcal{C}=\{C_1,\dots,C_\kappa\}$. Let $D_i$ be the starred diagram obtained from $u_{s,p}(L;\mathcal{C}-C_i)$ by adding a star to the region containing the tail of $C_i$. Then $\Omega_g(D_i)=0$ for all $i$, and we define
    \[
        \nabla^u_L(W) = \frac{1}{\kappa} \sum_{i=1}^\kappa \nabla_{D_i}(W).
    \]
   
\end{itemize}

The \textit{over-closure mock Alexander polynomial} $\nabla^o_L(W)$ of $L$ is defined analogously, by replacing $u_{s,p}(L;\mathcal{C}-C_i)$ with $o_{s,p}(L;\mathcal{C}-C_i)$.
\end{definition}

For planar linkoids, we can define a mock Alexander polynomial in exactly the same way, but without starring a region after forming the under-closures since the diagram already contains the point at infinity as a star. So in the case of planar linkoids, the diagram obtained from forming 
$\kappa-1$ under-closures is already admissible and its mock Alexander polynomial can be computed straight away.

\begin{example}
Let $L$ be the planar linkoid depicted in Figure \ref{fig:linkoid_polynomial}. Then $u_{s,p}(L;C_1)$ and $u_{s,p}(L;C_2)$ are depicted in Figure \ref{fig:linkoid_polynomial}, and we can compute $\nabla^u_L(W)$ from the potential matrices associated to these diagrams to find
\[
    \nabla^u_L(W) = \frac{1}{2}W^5 - \frac{1}{2}W^4 -W^3 +2W^2 + 2W -1 -W^{-1} + \frac{1}{2} W^{-2} + \frac{1}{2} W^{-3}
\]
\end{example}

\begin{figure}[ht]
    \centering
    \includegraphics[width=0.7\linewidth]{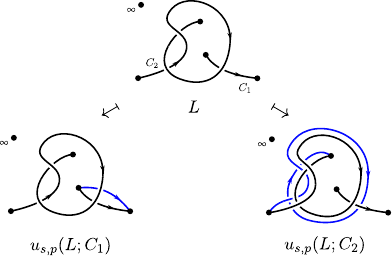}
    \caption{An example linkoid and its two admissible parallel shadow under-closures.}
    \label{fig:linkoid_polynomial}
\end{figure}

\begin{remark}
In the case of \textit{labelled} linkoid diagrams, i.e. diagrams whose components have been labelled, there is no concern of canonicity. In this case the constructions from Definition \ref{def:over-closures} can be applied straight away, since we can fix the choices made in these constructions using the component's labels.
\end{remark}

In Definition \ref{def:linkoid_MAP} we lower the obstruction by one too many, and then raise it back to zero by adding a single star. The reason we do so is that it reduces the number of choices to sum over: if all but one component is closed there are only $\kappa$ choices, where-as if alternatively all but two are closed there are $\binom{\kappa}{2}=\frac{k(k-1)}{2}$.

Although it has worse complexity in $\kappa$, namely quadratic rather than linear, this alternative approach to defining a mock Alexander polynomial for linkoids is also valid, and may result in invariants with distinguishing power different from $\nabla^o_L(W)$ and $\nabla^u_L(W)$. Similarly, in principle many constructions for replacing a linkoid diagram with an admissible generalized knotoid are possible. Each such construction gives rise to a polynomial invariant for linkoids via the state sum, and the distinguishing power of this invariant may of course be very different from that of $\nabla^o_L(W)$ and $\nabla^u_L(W)$. As an example we introduce the \textit{theta-closure} of a knotoidal component in a linkoid diagram:

\begin{definition}
Let $L$ be a linkoid, and let $C$ be a knotoidal component of $L$. Then we define the \textit{theta-closure} $\theta_C(L)$ by forming the arcs used to form both $u_{s,p}(L;C)$ and $o_{s,p}(L;C)$ and connecting these to the endpoints of $C$ to form two trivalent vertices. We assemble these vertices according to Figure \ref{fig:halfedges}, which specifies the cyclic order of half-edges chosen at each vertex. The result is a generalized knotoid diagram. Similarly if $S$ is a set of knotoidal components of $L$ then we define $\theta_S(L)$ as the generalized knotoid diagram obtained from forming the theta-closure of each component in $S$.
\end{definition}

\begin{figure}[ht]
    \centering
    \includegraphics[width=0.4\linewidth]{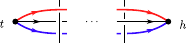}
    \caption{Cyclic order of half-edges at the trivalent points of $\theta_C(L)$.}
    \label{fig:halfedges}
\end{figure}

Since forming the theta-closure of a knotoidal component lowers $\Omega$ by two due to Lemma \ref{lm:edges}, the procedure of Definition \ref{def:linkoid_MAP} can only be applied to yield a polynomial invariant for linkoids with even obstruction, i.e. spherical linkoids with an even number of knotoidal components and planar knotoids with an odd number. In the spherical case, letting $\mathcal{K}$ denote the set of knotoidal components of $L$ we define the \textit{theta-closure mock Alexander polynomial} $\nabla^\theta_L(W)$ by
\[
    \nabla^\theta_L(W) = \frac{\kappa! \cdot \kappa}{(\frac{1}{2}\kappa!)^2(\kappa+2)} \sum_{\substack{S\subseteq \mathcal{K} \\ \lvert S\rvert = \frac{1}{2}(\kappa-2)}} \nabla_{\theta_S(L)} (W).
\]
Here the sum defining $\nabla^\theta_L(W)$ is divided by $\binom{\kappa}{\frac{1}{2}\kappa-1} = \frac{(\frac{1}{2}\kappa!)^2(\kappa+2)}{\kappa! \cdot \kappa}$, as this is the number of choices for $S$. Similarly if $L$ is a planar $(\kappa,\ell)$-linkoid with $\kappa$ odd we can define $\nabla^\theta_L(W)$ by the same formula, but with $S$ ranging over all subsets of $\mathcal{K}$ of cardinality $\frac{1}{2}(\kappa-1)$ and with the sum correspondingly normalized by $\binom{\kappa}{\frac{1}{2}(\kappa-1)}$.

\begin{remark}
Recalling Theorem \ref{thm:linkoidsandtheta} there is a clear motivation for forming theta-closures, since at the endpoints the theta-closure of a component resembles exactly its corresponding theta-curve. As such the theta-closure of a linkoid is not invariant under forbidden moves in all cases, unlike various closure operations.
\end{remark}

To extend $\nabla^\theta_L(W)$ to the case of odd obstruction, given linkoid diagram with odd $\Omega$ we can first form the shadow over- or under-closure of a single component to obtain an even obstruction, and then form theta-closures. Alternatively we can add an arc not forming a closure, but connecting two endpoints of different knotoidal components. To do so canonically, without the rest of the endpoints interfering, we embed this arc not in the sphere but in a handle attached to the sphere. 

More precisely, let $e_1,e_2$ be any two endpoints in a spherical linkoid $L$. Then we define the \textit{handle connection} $H_{e_1,e_2}(L)$ to be the linkoid diagram on the torus $\mathbb{T}^2$ formed from $L$ by adding a handle between $e_1$ and $e_2$ and running an arc along it, oriented from $e_1$ to $e_2$. If $e_1$ is a head and $e_2$ a tail endpoint, then the resulting arc is oriented. Otherwise it contains degree two vertices where the orientation is reversed, marking the locations of previous endpoints. If the endpoints $e_1,e_2$ lie in distinct regions, then it is easy to see that the graph on $\mathbb{T}^2$ obtained from $H_{e_1,e_2}(L)$ by seeing crossings as degree four vertices is cellularly embedded. This can be assumed without loss of generality using an $R2$ move; see Figure \ref{fig:R2_trick}. Using that $H_{e_1,e_2}(L)$ is cellularly embedded and hence obeys Euler's formula it is easy to deduce that $\Omega_g(H_{e_1,e_2}(L))= \Omega_g(L)+1$. As such it turns odd obstructions even and can be used along with the theta-closures to give a mock Alexander polynomial for odd-obstruction linkoids.

\begin{figure}[ht]
    \centering
    \includegraphics[width=0.35\linewidth]{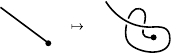}
    \caption{Using an $R2$ move to ensure two endpoints lie in distinct regions.}
    \label{fig:R2_trick}
\end{figure}

Although handle connections end up giving us generalized knotoids on the (punctured) torus rather than the plane or sphere, thereby taking us outside the scope of this paper, we mention them here briefly to show that in principle many constructions of generalized knotoids from linkoid diagrams are possible. We look forward to seeing further polynomial linkoid invariants derived via this scheme. \\

\textbf{Acknowledgements:} The authors would like to thank Ferenc Bencs and Pjotr Buys for helpful conversations, and the CIMPA-ICTP Research in Pairs Fellowship and CIRM for the support they provided during the writing of this paper.


\end{document}